\documentclass[12pt, twoside]{article}

\usepackage{amsmath,amsthm,mathtools,mathrsfs}
\usepackage{amssymb}

\usepackage{enumitem,empheq}

\usepackage{graphicx}

\usepackage[T1]{fontenc}

\newtheorem{thm}{Theorem}[section]
\newtheorem{cor}[thm]{Corollary}
\newtheorem{lem}[thm]{Lemma}

\theoremstyle{definition}

\newtheorem{rem}[thm]{Remark}

\numberwithin{equation}{section}

\newcommand{\Rd}{\mathbb{R}^d}

\newcommand{\dhh}{\partial_h}
\newcommand{\dt}{\partial_t}

\allowdisplaybreaks
\begin{document}

\baselineskip=17pt


\title{The evolution variational inequality\\ for weighted Wasserstein metrics\\ in non-convex bounded domains}

\author{Kyogo Murai}

\date{2026}

\maketitle

\renewcommand{\thefootnote}{}



\renewcommand{\thefootnote}{\arabic{footnote}}
\setcounter{footnote}{0}


\begin{abstract}
   In this paper, we establish the evolution variational inequality for the weighted Wasserstein distance,
   without assuming convexity of domains.
   Thanks to this evolution variational inequality,
   we can carry out some arguments with weighted Wasserstein metrics in not only convex but also non-convex domains.
   Therefore finally, we apply the evolution variational inequality to the minimizing movement in weighted Wasserstein metrics
   to obtain weak solutions of Keller--Segel systems and Cahn--Hilliard type equations
   in non-convex domains.
   The key point to remove the convexity assumption is a control of the boundary integral.
   To deal with the boundary integral, we use estimates for functions on the boundary, 
   the Sobolev trace embedding and the variant of Kato's inequality.
   Then, the boundary integral can be absorbed by good known terms.
\end{abstract}


\section{Introduction}

\quad The minimizing movement (or JKO scheme) is a method to prove existence of solutions to equations
which have gradient flow structures (energy structures) in suitable metric spaces.
The first work to consider this method in the Wasserstein space,
which is the complete metric space of probability measures,
was carried out by Jordan, Kinderlehrer and Otto in \cite{JKO}.
Here, the Wasserstein distance is defined by 
\begin{align}\label{eq1.1}
   \mathcal{W}_2(\mu_0,\mu_1)^2 
   \coloneq \inf_{\gamma\in \Gamma(\mu_0,\mu_1)}\int_{\mathbb{R}^d\times\mathbb{R}^d}|x-y|^2\, d\gamma(x,y)
   \quad \mathrm{for}\ \mu_0,\mu_1 \in \mathcal{P}_2(\mathbb{R}^d),
\end{align}
where $\mathcal{P}_2(\mathbb{R}^d)$ is the set of probability measures with the finite second moment
and $\Gamma(\mu_0,\mu_1)$ is the set of $\gamma \in \mathcal{P}(\mathbb{R}^d\times\mathbb{R}^d)$ satisfying
\begin{align*}
   \gamma(A\times\mathbb{R}^d) = \mu_0(A),\ \gamma(\mathbb{R}^d\times A) = \mu_1(A)\quad \mathrm{for\ all\ Borel\ set}\ A \subset \mathbb{R}^d.
\end{align*}
Thanks to this definition (the moment formulation),
we can use the perturbation $\mu_a$ of a measure $\mu$, 
which is needed when we consider the Euler--Lagrange equation in the minimizing movement.
Here $\mu_a$ is the push-forward measure of $\mu$ with
a map $T_a : \mathbb{R}^d \ni x \mapsto x +a\xi \in \mathbb{R}^d$ for $a>0$ and $\xi\in C_c^\infty(\mathbb{R}^d;\mathbb{R}^d)$,
that is, $\mu_a = T_a{\#}\mu$, 
where $T_a{\#}\mu$ is defined by $T_a{\#}\mu(A) = \mu(T_a^{-1}(A))$ for all Borel set $A\subset \mathbb{R}^d$.
Then this method can be also adapted for bounded domains
without assuming convexity by the zero extension of a measure $\mu$.
Hence the existence of solutions to equations which have the gradient flow structures in the Wasserstein space
are proved in not only whole space $\mathbb{R}^d$ but also bounded domains (we do not require convexity of domains, see \cite{Mi1,Mi2,KM}).

\quad On the other hand, in \cite{DNS}, a new metric space of measures is introduced, 
which is called the weighted Wasserstein space.
Here the weighted Wasserstein distance is formally defined by 
\begin{align*}
   W_m(\mu_0,\mu_1)^2 
   = \inf\left[\int_0^1\int_{\mathbb{R}^d} \frac{|\boldsymbol{\nu}(x,t)|^2}{m(\mu(x,t))}\, dx\, dt : (\mu,\boldsymbol{\nu}) \in CE(0,1;\mu_0\to\mu_1)\right],
\end{align*}
where $m : [0,M] \to [0,\infty)$ is a concave function ($M\in(0,\infty]$) and
$CE(0,1;\mu_0\to\mu_1)$ is the set of pairs of nonnegative Radon measures satisfying
the continuity equation in the weak sense, $\mu(\cdot,0) = \mu_0$ and $\mu(\cdot,1) = \mu_1$
(see \cite{DNS} for the precise definition of the weighted Wasserstein distance).
The weighted Wasserstein distance is the extension of the Wasserstein distance,
that is, if the function $m$ called mobility is linear ($m(r) = r$)
then the weighted Wasserstein distance coincides with the Wasserstein distance by the Benamou--Brenier theorem (\cite{BB}).
However, in contrast to the Wasserstein distance, 
the weighted Wasserstein distance does not have a moment representation such as \eqref{eq1.1},
so we cannot use the same way to consider the perturbation in the weighted Wasserstein space.

\quad For this reason, when we consider the Euler--Lagrange equation in the minimizing movement for the weighted Wasserstein space,
we employ another way called the flow interchange lemma introduced in \cite{MMS}.
According to the previous papers (\cite{LMS, M, Z, Zi}), convexity of domains is required to establish
the evolution variational inequality in the flow interchange lemma for the weighted Wasserstein space.
Now let us explain the flow interchange lemma.
First, we consider the minimizing movement as follows:
let $\Omega \subset \mathbb{R}^d$ be a bounded convex domain with the smooth boundary,
$X(\Omega)$ be a suitable function space,
$W_m$ be the weighted Wasserstein distance 
and $E : X(\Omega) \to (-\infty,\infty]$ be a proper lower semi-continuous functional.
Then for the initial data $u_0 \in X(\Omega)$ with $E(u_0) < \infty$ and $\tau > 0$,
define $\{u_\tau^k\}_{k} \subset X(\Omega)$ inductively by
\begin{align}\label{eq1.2}
   u_\tau^0 = u_0,\ u_\tau^{k+1} \in \mbox{argmin}\Psi_\tau^k\quad
   \mathrm{where}\ \Psi_\tau^k(v) = \frac{1}{2\tau}W_m(v,u_\tau^k)^2 + E(v)\quad \mathrm{for}\ v \in X(\Omega),
\end{align}
where $\mbox{argmin}\Phi$ denotes the set of minimizers for $\Phi$.
Next, let $F : X(\Omega) \to (-\infty,\infty]$ be a proper lower semi-continuous functional
and $S_t : \mbox{Dom}(F) \to \mbox{Dom}(F), t\geq0$ be a continuous semi-group, where $\mbox{Dom}(F) \coloneq \{v \in X(\Omega);F(v) < \infty\}$.
Assume that the Evolution Variational Inequality (EVI) for the weighted Wasserstein distance holds:
there exists $\lambda \in \mathbb{R}$ such that
\begin{align}\label{eq1.3}
   \frac{1}{2}\limsup_{h\downarrow0}\left[\frac{W_m(S_h(u),v)^2 - W_m(u,v)^2}{h}\right]
   + \frac{\lambda}{2}W_m(u,v)^2 \leq F(v) - F(u),
\end{align}
for all $u,v \in \mbox{Dom}(F)$ with $W_m(u,v) < \infty$.
Then since $u_\tau^k$ is the minimizer of $\Psi_\tau^{k-1}$ in $X(\Omega)$
and $S_h(u_\tau^k) \in X(\Omega)$, we have
\begin{align}\label{eq1.4}
   \frac{1}{2\tau}W_m(u_\tau^k,u_\tau^{k-1})^2 + E(u_\tau^k) 
   \leq \frac{1}{2\tau}W_m(S_h(u_\tau^k),u_\tau^{k-1})^2 + E(S_h(u_\tau^k))\quad \mathrm{for}\ h>0.
\end{align}
Combining this with the EVI \eqref{eq1.3}, we obtain
\begin{align}\label{eq1.5}
   -\limsup_{h\downarrow0}\frac{E(S_h(u_\tau^k)) - E(u_\tau^k)}{h}
   \leq -\frac{\lambda}{2}W_m(u_\tau^k,u_\tau^{k-1})^2 + F(u_\tau^{k-1}) - F(u_\tau^k).
\end{align}
This inequality formally implies that the derivative of $E$ at the minimizer $u_\tau^k$
can be controlled by the another functional $F$.
Thus, choosing $F$ suitably,
we can use the inequality \eqref{eq1.5} to obtain the Euler--Lagrange equation and to improve the regularity of the minimizer $u_\tau^k$.

\quad The key point of the flow interchange lemma is the EVI \eqref{eq1.3}
because the inequality \eqref{eq1.4} is automatically obtained by the minimizing scheme \eqref{eq1.2} if the semi-group $S_t$ exists.
Hence in \cite{LMS} (see also \cite[Theorem 2.2]{DS}), it is shown that the EVI for the weighted Wasserstein distance holds 
by proving the following differential inequality:
\begin{align}\label{eq1.6}
   \frac{1}{2}\partial_h\boldsymbol{A}^h(z) + \partial_zF(\rho^h(z)) 
   \leq - \lambda z \boldsymbol{A}^h(z)\quad \mathrm{for}\ z \in [0,1],\ h>0,
\end{align}
where $\rho^h(z) = S_{hz}(\rho(\cdot,z))$ for $\rho \in C^\infty(\overline{\Omega}\times[0,1])$,
\begin{align*}
   \boldsymbol{A}^h(z) = \int_{\Omega} m(\rho^h(z))|\nabla \phi^h(z)|^2\, dx,
\end{align*}
and $\phi^h(z)$ satisfies $\partial_z \rho^h(z) + \nabla\cdot(m(\rho^h(z))\nabla\phi^h(z)) = 0$ in $\Omega$
with a suitable boundary condition.
Note that the integral of $\boldsymbol{A}^h$ with respect to $z\in[0,1]$
can approximate the weighted Wasserstein distance (see \cite[Proposition 2.2]{LMS} and \cite[Lemma 2.8]{M}).
Then, in the proof of the inequality \eqref{eq1.6} in \cite{LMS},
the convexity of the domain $\Omega$ is required 
to control a boundary integral.
Hence up to now, the minimizing movement in the weighted Wasserstein space can be 
applied only in bounded convex domains (see also \cite{Z,Zi,M}).

\quad The purpose of this paper is to prove the differential inequality \eqref{eq1.6}
without assuming the convexity of the domain $\Omega$.
The obstacle is the boundary integral
\begin{align*}
   \int_{\partial\Omega} m(\rho^h(z))\nabla|\nabla\phi^h(z)|^2\cdot\boldsymbol{n}\, dS,
\end{align*}
which appears in the proof of \eqref{eq1.6},
where $\boldsymbol{n}$ is the outer unit normal vector to $\partial\Omega$.
If the domain is convex, then the above boundary integral is nonpositive,
that is, it can be estimated by zero from above.
The basic idea to control the boundary integral is inspired by \cite{ISY},
so we use the estimate of the term $\nabla|\nabla\phi^h(z)|^2\cdot\boldsymbol{n}$ on the boundary (see Lemma \ref{lem2.1}),
the Sobolev trace embedding (Lemma 2.2) and the fractional Gagliardo--Nirenberg inequality (Lemma \ref{lem2.4}),
namely we have
\begin{align*}
   \int_{\partial\Omega} m(\rho^h(z))\nabla|\nabla\phi^h(z)|^2\cdot\boldsymbol{n}\, dS
   \leq C\left(\left\|\nabla\left[(m(\rho^h(z)))^{\frac{1}{2}}|\nabla\phi^h(z)|\right]\right\|_{L^2(\Omega)}^2 + \boldsymbol{A}^h(z)\right).
\end{align*}
Of cource, the considered functions (the inequality) are different from the ones in \cite{ISY},
then, we need to control the additional term
\begin{align}\label{eq1.7}
   \int_{\Omega} m(\rho^h(z))|\nabla|\nabla\phi^h(z)||^2\, dx,
\end{align}
which comes from the first term on the right-hand side.
On the other hand, in the proof of \cite{LMS}, the term 
\begin{align}\label{eq1.8}
   -\int_\Omega m(\rho^h(z))|\nabla^2\phi^h(z)|^2\, dx
\end{align}
appears and is estimated by zero from above (this negative term did not used effectively for the estimate),
where $\nabla^2f$ denotes the Hessian matrix of $f$ and 
\begin{align*}
   |\nabla^2 f| = \left(\sum_{i,j=1}^d \left(\partial_{ij}f\right)^2\right)^{\frac{1}{2}}.
\end{align*}
Thus, combining the variant of Kato's inequality (Lemma \ref{lem2.5}) with \eqref{eq1.7},
the additional term \eqref{eq1.7} can be absorbed by the negative term \eqref{eq1.8}.

\quad This paper is organized as follows.
In Section 2, we collect some estimates to control the boundary integral.
In Section 3, at first we introduce some assumptions of the mobility $m$
and some settings. Then we present our main theorem (the differential inequality) and its proof.
Finally, in Section 4, 
we apply our main theorem to the minimizing movement in the weighted Wasserstein space.
In particular, we show existence of weak solutions to 
Keller--Segel systems and Cahn--Hilliard type equations in a bounded non-convex domain
by the minimizing movement.


\section{Preliminaries}

\quad Let $d \in \mathbb{N}$ and 
$\Omega$ be a bounded domain in $\mathbb{R}^d$ with the smooth boundary.
The following lemmas have an imprtant role to remove the convexity assumption of domains.

\begin{lem}[{\cite[Lemma 4.2]{MS}}]\label{lem2.1}
   There exists a constant $C_\Omega>0$ depending on $\Omega$ such that
   \begin{align*}
      \nabla (|\nabla w|^2)\cdot\boldsymbol{n} \leq C_\Omega|\nabla w|^2\quad \mathrm{on}\ \partial\Omega,
   \end{align*}
   for all $w \in C^2(\bar{\Omega})$ satisfying $\nabla w\cdot\boldsymbol{n} = 0$ on $\partial\Omega$. 
\end{lem}

\begin{lem}[{\cite[Proposition 4.22(ii) and Proposition 4.24(i)]{HT}}]\label{lem2.2} 
   Let $s>\frac{1}{2}$. 
   Then there exists a constant $C_0>0$ such that 
   \begin{align*}
      \|f\|_{L^2(\partial\Omega)}
      \leq C_0\|f\|_{W^{s,2}(\Omega)}\quad \mathrm{for}\ f \in W^{s,2}(\Omega).
   \end{align*}
\end{lem}

The next lemma is a fractional version of the Gagliardo--Nirenberg inequality
(see, for instance, \cite{BM} for more general cases).

\begin{lem}[fractional Gagliardo--Nirenberg inequality]\label{lem2.4}
   Let $0<s<1$.
   Then there exist constants $C_1, C_2>0$ such that for all $f \in H^{1}(\Omega)$,
   \begin{align*}
      \|f\|_{W^{s,2}(\Omega)} \leq C_1\|\nabla f\|_{L^2(\Omega)}^s\|f\|_{L^2(\Omega)}^{1-s} + C_2\|f\|_{L^2(\Omega)}.
   \end{align*}
\end{lem}

The following lemma is the variant of Kato's inequality.

\begin{lem}\label{lem2.5}
   Let $w \in C^2(\Omega)$. Then it holds
   \begin{align*}
      |\nabla|\nabla w|| \leq |\nabla^2 w|\quad\mathrm{a.e.\ in}\ \Omega,
   \end{align*}
   where 
   \begin{align*}
      |\nabla^2 w| = \left(\sum_{i,j=1}^d \left(\partial_{ij}w\right)^2\right)^{\frac{1}{2}}.
   \end{align*}
\end{lem}

\begin{proof}
   First, since $w \in C^2(\Omega)$, $|\nabla w|$ is locally Lipschitz continuous in $\Omega$
   and thus differentiable in $\Omega\setminus\mathcal{N}$ for some null set $\mathcal{N}$.
   Fix $x \in \Omega\setminus\mathcal{N}$.
   If $|\nabla w(x)| \ne 0$ then it follows from the chain rule that
   \begin{align*}
      \nabla|\nabla w(x)| = \frac{\nabla^2 w(x)\nabla w(x)}{|\nabla w(x)|},
   \end{align*}
   where $\nabla^2 w(x)$ is the Hessian matrix of $w$ at $x$.
   Thus by the Cauchy--Schwarz inequality, we obtain
   \begin{align*}
      |\nabla|\nabla w(x)||^2 
      &= \frac{1}{|\nabla w(x)|^2}\sum_{j=1}^d\left(\sum_{i=1}^d(\partial_{ij}w(x))(\partial_i w(x))\right)^2\\
      &\leq \frac{1}{|\nabla w(x)|^2}\left(\sum_{j=1}^d\sum_{i=1}^d(\partial_{ij}w(x))^2\right)\left(\sum_{i=1}^d(\partial_i w(x))^2\right)\\
      &= \frac{1}{|\nabla w(x)|^2}|\nabla^2 w(x)|^2|\nabla w(x)|^2
      = |\nabla^2 w(x)|^2.
   \end{align*}
   On the other hand, 
   if $|\nabla w(x)| = 0$ then 
   we have $\partial_i w(x) = 0$ for all $i=1,\cdots,d$.
   Let $e_i = (0,\cdots,0,1,0,\cdots,0)$ be the $i$-th standard basis vector in $\mathbb{R}^d$ 
   and $h \in \mathbb{R}$ be a small real number.
   Then the assumption $|\nabla w(x)| = 0$ and 
   the differentiability of $\partial_j w(x)$ for $j = 1,\cdots, d$ yield
   \begin{align*}
      \left(\frac{|\nabla w(x+he_i)| - |\nabla w(x)|}{h}\right)^2
      &= \left|\frac{\nabla w(x+he_i)}{h}\right|^2\\
      &= \sum_{j=1}^d\left(\frac{\partial_j w(x+he_i)}{h}\right)^2\\
      &\to \sum_{j=1}^d(\partial_{ij}w(x))^2\quad \mathrm{as}\ h\to0.
   \end{align*}
   Therefore we conclude
   \begin{align*}
      |\nabla |\nabla w(x)||^2 \leq \sum_{i=1}^d\sum_{j=1}^d(\partial_{ij}w(x))^2 = |\nabla^2 w(x)|^2.
   \end{align*}
   The proof is completed.
\end{proof}

\section{Main theorem}
\quad In this section, we first introduce some conditions of
the mobility $m$, and 
define the functional $F$ and the semi-group $S_t$
in order to prove the differential inequality \eqref{eq1.6}.

\subsection{Assumptions of the mobility}

\quad Let $M \in (0,\infty]$.
When $M=\infty$, 
let $m : [0,\infty) \to [0,\infty)$
be a smooth concave function satisfying $m(0) = 0$.
For $\varepsilon>0$, define the approximated mobility $m_\varepsilon: [0,\infty) \to [0,\infty)$ by
\begin{align*}
   m_\varepsilon(r) \coloneq m(r+\varepsilon).
\end{align*}
Then assume that the approximated mobility satisfies the following conditions:
\begin{align}\label{M-LSC}
   \begin{cases}
      \displaystyle
      \sup_{r\geq0}|m_\varepsilon^\prime(r)| < \infty,\\
      \displaystyle
      \sup_{r\geq0}|m_\varepsilon^{\prime\prime}(r)m_\varepsilon(r)| < \infty,
   \end{cases}\tag{M-LSC}
\end{align}
($m_\varepsilon$ is Lipschitz and $m_\varepsilon^2$ is semi-convex, see also \cite[Section 1.1]{LMS}),
and additionally
\begin{align}\label{M-A}
   \sup_{r\geq0}\frac{|m_\varepsilon^\prime(r)|^2}{|m_\varepsilon^{\prime\prime}(r)m_\varepsilon(r)|} < \infty.\tag{M-A}
\end{align}
Note that if 
\begin{align*}
   \sup_{r\geq0}\frac{|m_\varepsilon^\prime(r)|^2}{|m_\varepsilon^{\prime\prime}(r)m_\varepsilon(r)|} = 0,
\end{align*}
then $m(r) = 0$ for all $r\geq0$ by $m(0) = 0$.
Examples of mobility satisfying the above conditions are the following:
\begin{align*}
   m(r) = r^\alpha\ (0<\alpha<1),\quad m(r) = \log(1+r),\quad m(r) = r^\alpha(r+c)^\beta\ (0<\alpha+\beta\leq 1, c>0).
\end{align*}
When $M<\infty$, we consider the specific mobility of the form
\begin{align*}
   m(r) = [r(M-r)]^\alpha\ (0<\alpha\leq 1),
\end{align*}
and define the approximated moblity $m_\varepsilon: [0,M] \to [0,\infty)$ by
\begin{align*}
   m_\varepsilon(r) = [(r+\varepsilon)(M-r+\varepsilon)]^\alpha.
\end{align*}
One can confirm that $m_\varepsilon$ satisfies conditions \eqref{M-LSC} and \eqref{M-A}.

\begin{rem}
   The assumptions \eqref{M-LSC} and \eqref{M-A} are required to get the finiteness of coefficient
   $\lambda$ in the EVI \eqref{eq1.3} (or $\lambda_\delta$ in Theorem \ref{thm3.1}).
   If the domain is convex, then it is enough to assume \eqref{M-LSC} (see \cite{LMS}).
   However, if the domain is non-convex, then we need the additional condition \eqref{M-A} 
   in order to control more terms which come from the boundary integral.
   On the other hand, as we see in Section 4, 
   the mobility $m$ (approximated mobility $m_\varepsilon$) in the applications
   satisfies the condition \eqref{M-A}.
\end{rem}

The following lemma shows properties of approximation from above
(see also \cite[Lemma 5.1]{LMS} for approximation from below).

\begin{lem}\label{lem4.6}
   Let $m$ be the mobility satisfying the above conditions.
   Then the approximated mobility $m_\varepsilon$ is smooth and concave,
   and satisfies the pointwise bound $m \leq m_\varepsilon$.
   Moreover, $m_\varepsilon$ converges monotonically and uniformly to $m$ as $\varepsilon \to 0$.
\end{lem}

\begin{proof}
   The smoothness and concavity of $m_\varepsilon$ are obvious 
   because of its definition, the smoothness and concavity of $m$.

   When $M<\infty$, since the mobility is $m(r) = [r(M-r)]^\alpha\ (0<\alpha\leq 1)$
   and the approximated mobility is $m_\varepsilon(r) = [(r+\varepsilon)(M-r+\varepsilon)]^\alpha$,
   it is easy to get
   \begin{align*}
      m(r) 
      &= [r(M-r)]^\alpha\\
      &\leq [(r+\varepsilon_1)(M-r+\varepsilon_1)]^\alpha = m_{\varepsilon_1}(r)\\
      &\leq [(r+\varepsilon_2)(M-r+\varepsilon_2)]^\alpha = m_{\varepsilon_2}(r)\quad \mathrm{for}\ 0< \varepsilon_1< \varepsilon_2.
   \end{align*}
   In addition, due to the inequality
   \begin{align*}
      a^\alpha - b^\alpha \leq (a-b)^\alpha\quad \mathrm{for}\ 0\leq b\leq a,\ 0<\alpha\leq 1,
   \end{align*}
   we obtain
   \begin{align*}
      0&\leq m_{\varepsilon}(r) - m(r)\\
      &= \bigg([(r+\varepsilon)(M-r+\varepsilon)]^\alpha - [(r+\varepsilon)(M-r)]^\alpha\bigg)
       + \bigg([(r+\varepsilon)(M-r)]^\alpha - [r(M-r)]^\alpha\bigg)\\
      &\leq (r+\varepsilon)^\alpha\varepsilon^\alpha + \varepsilon^\alpha(M-r)^\alpha\\
      &\leq \left[(M+\varepsilon)^\alpha + M^\alpha\right]\varepsilon^\alpha\quad \mathrm{for\ all}\ r\in [0,M],
   \end{align*}
   which implies the uniform convergence as $\varepsilon\to0$.

   When $M=\infty$, 
   since $m$ is nonnegative and concave, that is, $m$ is monotonically increasing,
   it holds that
   \begin{align*}
      m(r) 
      &\leq m(r+\varepsilon_1) = m_{\varepsilon_1}(r)\\
      &\leq m(r+\varepsilon_2) = m_{\varepsilon_2}(r)\quad \mathrm{for}\ 0< \varepsilon_1 < \varepsilon_2.
   \end{align*}
   Moreover, since $m(0) = 0$ and $m$ is concave, that is, the function
   \begin{align*}
      s \mapsto \frac{m(\cdot+s) - m(\cdot)}{s}
   \end{align*}
   is monotonically decreasing, we have
   \begin{align*}
      0\leq m_\varepsilon(r) - m(r) &= \frac{m(r+\varepsilon) - m(r)}{\varepsilon}\varepsilon\\
      &\leq \frac{m(\varepsilon) - m(0)}{\varepsilon}\varepsilon = m(\varepsilon) \to 0\quad \mathrm{as}\ \varepsilon \to 0
   \end{align*}
   for all $r\in [0,\infty)$.
   The proof is completed.
\end{proof}

\subsection{Settings for main theorem}

\quad Let $\Omega$ be a bounded domain in $\mathbb{R}^d$ with the smooth boundary (we do not require the convexity).
Let $m$ and $m_\varepsilon$ be mobility and approximated mobility respectively introduced in Section 3.1.
Fix $\varphi \in C^\infty(\overline{\Omega})$ with $\nabla\varphi\cdot\boldsymbol{n} = 0$ on $\partial\Omega$ and $\delta>0$.
Let $\rho \in C^\infty(\bar{\Omega}\times[0,1])$ be a nonnegative smooth function and fix $z\in[0,1]$. 
For the smooth function $\rho(\cdot,z) \eqcolon \rho(z)$, by the standard argument (see, for instance, \cite[Propositions 4.1 and 4.2]{M}), 
we have a solution $S_t\rho(z)$ satisfying
\begin{align}
   &\bullet S_t\rho(z) \in C^\infty(\overline{\Omega}\times[0,\infty)),\notag\\
   &\bullet \partial_t (S_t\rho(z)) 
    = \delta\Delta (S_t\rho(z)) + \nabla\cdot(m_\varepsilon(S_t\rho(z))\nabla\varphi)
    \quad \mathrm{in}\ \Omega\times[0,\infty),\label{eq46}\\
   &\bullet \nabla S_t\rho(z)\cdot\boldsymbol{n} = 0\quad \mathrm{on}\ \partial\Omega\times(0,\infty),\notag\\
   &\bullet S_0\rho(z) = \rho(z)\quad \mathrm{in}\ \Omega,\notag\\
   &\bullet S_t\rho(z) \geq 0,\ \|S_t\rho(z)\|_{L^1(\Omega)} = \|\rho(z)\|_{L^1(\Omega)}
    \quad \mathrm{for}\ t \in [0,\infty),\notag\\
   &\bullet S_t\rho(z) \leq M\ \mathrm{if}\ M<\infty.\notag
\end{align}

In addition, define $\rho^h(z) \coloneq S_{hz}\rho(z)$ for $h\in(0,1)$.
Note that $\rho^h(z)$ is $z$-differentiable because of the smoothness of $\rho(z)$ and $S_t\rho(z)$.
Since $\|\rho^h(z)\|_{L^1(\Omega)} = \|\rho(z)\|_{L^1(\Omega)}$ for all $z\in[0,1]$ and $h>0$, we have
\begin{align*}
   \int_{\Omega} \partial_z\rho^h(z)\, dx = 0.
\end{align*}
Hence for fixed $z\in[0,1]$,  
we can find a unique weak solution $\phi^h(z) \in H^1(\Omega)$ of
\begin{empheq}[left={\empheqlbrace}]{alignat=2}\label{ee2}
      &\partial_z \rho^h(z) 
      = -\nabla\cdot(m_{\varepsilon}(\rho^h(z))\nabla\phi^h(z)) &\quad & \mathrm{in}\ \Omega,\\
      &\nabla\phi^h(z)\cdot\boldsymbol{n} = 0 &\quad & \mathrm{on}\ \partial\Omega,\notag
\end{empheq}
and the elliptic regularity theorem yields that $\phi^h \in C^\infty(\overline{\Omega}\times[0,1])$
and $\phi^h$ is $h$-differentiable.
We set up
\begin{equation*}
   \boldsymbol{A}^h(z) 
   \coloneq \int_{\Omega} m_{\varepsilon}(\rho^h(z))|\nabla\phi^h(z)|^2\, dx,
\end{equation*}
and define the functional 
$F : L^1\cap\mathcal{P}(\Omega) \to (-\infty, \infty]$ by 
\begin{align*}
   F(u) 
   \coloneq \int_{\Omega} u\varphi\, dx + \delta \int_{\Omega} U_\varepsilon(u)\, dx,
\end{align*}
where $\mathcal{P}(\Omega)$ is the set of Borel probability measures on $\Omega$ and 
$U_\varepsilon : [0,M] \to \mathbb{R}$ satisfies
\begin{align}\label{eq3.3}
   U_\varepsilon^{\prime\prime}(r)m_\varepsilon(r) = 1\quad \mathrm{for}\ 0\leq r \leq M.
\end{align}
Note that since $m_\varepsilon$ is smooth, $U_\varepsilon$ is also smooth.

\subsection{Main theorem and its proof}

\quad When the domain is convex, 
the following differential inequality \eqref{lem43} is established in \cite{LMS}.
The difference of our proof from the one in \cite{LMS} is to control the boundary integral ($\mathrm{J}_5$ in the proof).

\begin{thm}\label{thm3.1}
   Let $\sigma>0$ be a sufficiently small, which is determined by $\Omega$ and $\varepsilon$.
   Under the settings of Sections 3.1 and 3.2,
   the following differential inequality holds:
   \begin{equation}\label{lem43}
      \frac{1}{2}\dhh\boldsymbol{A}^h(z) 
      + \partial_zF(\rho^h(z)) 
      \leq -\lambda_\delta z \boldsymbol{A}^h(z)
      \quad \mathrm{for}\ z\in[0,1]\ \mathrm{and}\ h\in (0,1),
   \end{equation}
   where 
   \begin{align*}
      \lambda_{\delta} 
      &\coloneq -\frac{1}{\delta}\|\nabla\varphi\|_{L^\infty(\Omega)}^2\sup_{0<r<M}(m_\varepsilon(r)|m_\varepsilon^{\prime\prime}(r)|) 
       - \|\nabla^2\varphi\|_{L^\infty(\Omega)}\sup_{0<r<M}|m_\varepsilon^\prime(r)|
       - \frac{\delta C}{\sigma} \leq 0,
   \end{align*}
   $C > 0$ is a constant depending on $\Omega$ and
   $$\|\nabla^2\varphi\|_{L^\infty(\Omega)} 
   = \left\|\left(\sum_{i,j=1}^d \left(\partial_{ij}\varphi\right)^2\right)^{\frac{1}{2}}\right\|_{L^\infty(\Omega)}.$$
\end{thm}

\begin{proof}
   Fix $z\in [0,1]$ and $h\in(0,1)$.
   First by \eqref{eq46} and \eqref{ee2}, we have
   \begin{align}\label{eq1}
      \dhh\rho^h(z) 
      &= z(\partial_tS_t\rho(z))_{t=hz}\notag\\
      &= \delta z \Delta\rho^h(z) + z\nabla\cdot(m_\varepsilon(\rho^h(z))\nabla\varphi),
   \end{align} 
   and
   \begin{align}\label{eq2}
      \partial_z\dhh\rho^h(z) 
      &= \delta\Delta\rho^h(z) 
       + \nabla\cdot(m_\varepsilon(\rho^h(z))\nabla\varphi)\notag\\
      &\quad - \delta z\Delta[\nabla\cdot(m_\varepsilon(\rho^h(z))\nabla\phi^h(z))]\notag\\
      &\quad - z\nabla\cdot[\nabla\cdot(m_\varepsilon(\rho^h(z))\nabla\phi^h(z))m_\varepsilon^\prime(\rho^h(z))\nabla\varphi].
   \end{align}
   Then, it follows from \eqref{ee2}, $\nabla\phi^h(z)\cdot\boldsymbol{n} = 0$ on $\partial\Omega$ 
   and the integration by parts that
   \begin{align*}
      \partial_zF(\rho^h(z)) 
      &= \int_{\Omega} \partial_z\rho^h(z)\varphi\, dx + \delta\int_{\Omega} U_\varepsilon^\prime(\rho^h(z))\partial_z\rho^h(z)\, dx\\
      &= -\int_{\Omega} \nabla\cdot(m_\varepsilon(\rho^h(z))\nabla\phi^h(z))\varphi\, dx\\
      &\quad -\delta\int_{\Omega} U_\varepsilon^\prime(\rho^h(z))\nabla\cdot(m_\varepsilon(\rho^h(z))\nabla\phi^h(z))\, dx\\
      &= \int_{\Omega} m(\rho^h(z))\nabla\varphi\cdot\nabla\phi^h(z)\, dx\\
      &\quad + \delta\int_{\Omega} U_\varepsilon^{\prime\prime}(\rho^h(z))m_\varepsilon(\rho^h(z))\nabla\phi^h(z)\cdot\nabla\rho^h(z)\, dx.
   \end{align*}
   Since 
   $U_\varepsilon^{\prime\prime}(\rho^h(t))m_\varepsilon(\rho^h(t)) = 1$
   due to the property \eqref{eq3.3},
   then we obtain
   \begin{align}\label{eq3}
      \partial_zF(\rho^h(z)) 
      &= \int_{\Omega} m_\varepsilon(\rho^h(z))\nabla\varphi\cdot\nabla\phi^h(z)\, dx
       + \delta\int_{\Omega} \nabla\rho^h(z)\cdot\nabla\phi^h(z)\, dx.
   \end{align}
   On the other hand, the direct calculation yields
   \begin{align*}
      \frac{1}{2}\dhh\boldsymbol{A}^h(z) 
      &= \frac{1}{2}\int_{\Omega} \partial_h \{m_\varepsilon(\rho^h(z))\}|\nabla\phi^h(z)|^2\, dx
       + \int_{\Omega} m_\varepsilon(\rho^h(z))\nabla\phi^h(z)\cdot\partial_h\nabla\phi^h(z)\, dx\\
      &= \int_{\Omega} [\partial_h \{m_\varepsilon(\rho^h(z))\}\nabla\phi^h(z) + m_\varepsilon(\rho^h(z))\partial_h\nabla\phi^h(z)]\cdot\nabla\phi^h(z)\, dx\\
      &\quad - \frac{1}{2}\int_{\Omega} \partial_h \{m_\varepsilon(\rho^h(z))\}|\nabla\phi^h(z)|^2\, dx\\
      &= -\frac{1}{2}\int_{\Omega} \dhh \{m_\varepsilon(\rho^h(z))\}|\nabla\phi^h(z)|^2\, dx
       + \int_{\Omega} \dhh[m_\varepsilon(\rho^h(z))\nabla\phi^h(z)]\cdot\nabla\phi^h(z)\, dx.
   \end{align*}
   Here, since $\nabla\phi^h(z)\cdot\boldsymbol{n} = 0$ on $\partial\Omega$ for all $h\in (0,1)$,
   it holds that
   \begin{align*}
      \partial_h\nabla\phi^h(z)\cdot\boldsymbol{n} = 0\quad \mathrm{on}\ \partial\Omega.
   \end{align*}
   Thus using the integration by parts, we have
   \begin{align}\label{eq4}
      \frac{1}{2}\dhh\boldsymbol{A}^h(z)
      &= -\frac{1}{2}\int_{\Omega} \dhh\rho^h(z)m_\varepsilon^\prime(\rho^h(z))|\nabla\phi^h(z)|^2\, dx
       - \int_{\Omega} \dhh[\nabla\cdot(m_\varepsilon(\rho^h(z))\nabla\phi^h(z))]\phi^h(z)\, dx\notag\\
      &= -\frac{1}{2}\int_{\Omega} \dhh\rho^h(z)m_\varepsilon^\prime(\rho^h(z))|\nabla\phi^h(z)|^2\, dx
       + \int_{\Omega} (\dhh\partial_z\rho^h(z))\phi^h(z)\, dx\notag\\
      &\eqcolon \mathrm{I}_1 + \mathrm{I}_2.
   \end{align}
   First, we calculate $\mathrm{I}_1$.
   It follows from
   \eqref{eq1}, $\nabla\rho^h(z)\cdot\boldsymbol{n} = 0, \nabla\varphi\cdot\boldsymbol{n} = 0$ on $\partial\Omega$ 
   and integration by parts that
   \begin{align*}
      \mathrm{I}_1
      &= -\frac{z}{2}\int_{\Omega} [\delta\Delta\rho^h(z) + \nabla\cdot(m_\varepsilon(\rho^h(z))\nabla\varphi)]m_\varepsilon^\prime(\rho^h(z))|\nabla\phi^h(z)|^2\, dx\\
      &= \frac{z}{2}\int_{\Omega} [\delta\nabla\rho^h(z) + m_\varepsilon(\rho^h(z))\nabla\varphi]\cdot\nabla[m_\varepsilon^\prime(\rho^h(z))|\nabla\phi^h(z)|^2]\, dx\\
      &= \frac{z\delta}{2}\int_{\Omega} m_\varepsilon^{\prime\prime}(\rho^h(z))|\nabla\rho^h(z)|^2|\nabla\phi^h(z)|^2\, dx\\
      &\quad + \frac{z\delta}{2}\int_{\Omega} m_\varepsilon^\prime(\rho^h(z))\nabla\rho^h(z)\cdot\nabla(|\nabla\phi^h(z)|^2)\, dx\\
      &\quad + \frac{z}{2}\int_{\Omega} m_\varepsilon(\rho^h(z))m_\varepsilon^{\prime\prime}(\rho^h(z))\nabla\varphi\cdot\nabla\rho^h(z)|\nabla\phi^h(z)|^2\, dx\\
      &\quad + z\int_{\Omega} m_\varepsilon(\rho^h(z))m_\varepsilon^\prime(\rho^h(z))\nabla\varphi\cdot\nabla\left(\frac{1}{2}|\nabla\phi^h(z)|^2\right)\, dx.
   \end{align*}
   Using $m_\varepsilon^\prime(\rho^h(z))\nabla\rho^h(z) = \nabla m_\varepsilon(\rho^h(z))$ and
   integration by parts again, we obtain
   \begin{align}\label{eq5}
      \mathrm{I}_1
      &= \frac{z\delta}{2}\int_{\Omega} m_\varepsilon^{\prime\prime}(\rho^h(z))|\nabla\rho^h(z)|^2|\nabla\phi^h(z)|^2\, dx\notag\\
      &\quad + \frac{z\delta}{2}\int_{\partial\Omega} m_\varepsilon(\rho^h(z))\nabla|\nabla\phi^h(z)|^2\cdot\boldsymbol{n}\, dS\notag\\
      &\quad - z\delta\int_{\Omega} m_\varepsilon(\rho^h(z))\Delta\left(\frac{1}{2}|\nabla\phi^h(z)|^2\right)\, dx\notag\\
      &\quad + \frac{z}{2}\int_{\Omega} m_\varepsilon(\rho^h(z))m_\varepsilon^{\prime\prime}(\rho^h(z))\nabla\varphi\cdot\nabla\rho^h(z)|\nabla\phi^h(z)|^2\, dx\notag\\
      &\quad + z\int_{\Omega} m_\varepsilon(\rho^h(z))m_\varepsilon^\prime(\rho^h(z))\nabla\varphi\cdot\nabla\left(\frac{1}{2}|\nabla\phi^h(z)|^2\right)\, dx.
   \end{align}
   Next for $\mathrm{I}_2$, we infer from \eqref{eq2} and integration by parts that
   \begin{align*}
      \mathrm{I}_2
      &= -\delta\int_{\Omega} \nabla\rho^h(z)\cdot\nabla\phi^h(z)\, dx
       - \int_{\Omega} m_\varepsilon(\rho^h(z))\nabla\varphi\cdot\nabla\phi^h(z)\, dx\\
      &\quad + z\int_{\Omega} \nabla\cdot[m_\varepsilon(\rho^h(z))\nabla\phi^h(z)]m_\varepsilon^\prime(\rho^h(z))\nabla\varphi\cdot\nabla\phi^h(z)\, dx\\
      &\quad -z\delta\int_{\Omega} \Delta[\nabla\cdot\{m_\varepsilon(\rho^h(z))\nabla\phi^h(z)\}]\phi^h(z)\, dx.
   \end{align*}
   Here, since
   $\nabla \rho^h(z)\cdot\boldsymbol{n} = 0$ on $\partial\Omega\times(0,\infty)$ for all $z\in [0,1]$ and
   $\rho^h(z)$ is smooth with respect to $z\in[0,1]$, it follows 
   \begin{align*}
      \nabla \partial_z\rho^h(z)\cdot\boldsymbol{n} = \partial_z\nabla\rho^h(z)\cdot\boldsymbol{n} = 0\quad \mathrm{on}\ \partial\Omega\times(0,\infty).
   \end{align*}
   Thus we infer from $\partial_z\rho^h(z) = -\nabla\cdot\{m_\varepsilon(\rho^h(z))\nabla\phi^h(z)\}$ in $\Omega$ 
   and integration by parts that 
   \begin{align*}
      z\delta\int_{\Omega} \Delta[-\nabla\cdot\{m_\varepsilon(\rho^h(z))\nabla\phi^h(z)\}]\phi^h(z)\, dx
      &= z\delta\int_{\Omega} \Delta[\partial_z\rho^h(z)]\phi^h(z)\, dx\\
      &= -z\delta\int_{\Omega} \nabla\partial_z\rho^h(z)\cdot\nabla\phi^h(z)\, dx\\
      &= z\delta\int_{\Omega} \partial_z\rho^h(z)\Delta\phi^h(z)\, dx\\
      &= -z\delta\int_{\Omega} \nabla\cdot[m_\varepsilon(\rho^h(z))\nabla\phi^h(z)]\Delta\phi^h(z)\, dx.
   \end{align*}
   Hence using \eqref{eq3}, integration by parts 
   and $\nabla\phi^h(z)\cdot\boldsymbol{n} = 0$ on $\partial\Omega$, we obtain
   \begin{align}\label{eq53}
      \mathrm{I}_2
      &= -\partial_zF(\rho^h(z))\notag\\
      &\quad - z\int_{\Omega} m_\varepsilon(\rho^h(z))\nabla\phi^h(z)\cdot\nabla[m_\varepsilon^\prime(\rho^h(z))\nabla\varphi\cdot\nabla\phi^h(z)]\, dx\notag\\
      &\quad + z\delta\int_{\Omega} m_\varepsilon(\rho^h(z))\nabla\phi^h(z)\cdot\nabla(\Delta\phi^h(z))\, dx\notag\\
      &= -\partial_zF(\rho^h(z))\notag\\
      &\quad - z\int_{\Omega} m_\varepsilon(\rho^h(z))m_\varepsilon^{\prime\prime}(\rho^h(z))[\nabla\phi^h(z)\cdot\nabla\rho^h(z)][\nabla\varphi\cdot\nabla\phi^h(z)]\, dx\notag\\
      &\quad - z\int_{\Omega} m_\varepsilon(\rho^h(z))m_\varepsilon^\prime(\rho^h(z))\nabla\phi^h(z)\cdot\nabla[\nabla\varphi\cdot\nabla\phi^h(z)]\, dx\notag\\
      &\quad + z\delta\int_{\Omega} m_\varepsilon(\rho^h(z))\nabla\phi^h(z)\cdot\nabla(\Delta\phi^h(z))\, dx.
   \end{align}
   Combining \eqref{eq4} with \eqref{eq5} and \eqref{eq53}, we have
   \begin{align*}
      &\frac{1}{2}\dhh\boldsymbol{A}^h(z) + \partial_zF(\rho^h(z))\\
      &= \frac{z\delta}{2}\int_{\Omega} m_\varepsilon^{\prime\prime}(\rho^h(z))|\nabla\rho^h(z)|^2|\nabla\phi^h(z)|^2\, dx\\
      &\quad + z\delta\int_{\Omega} m_\varepsilon(\rho^h(z))\left[-\Delta\left(\frac{1}{2}|\nabla\phi^h(z)|^2\right) + \nabla\phi^h(z)\cdot\nabla(\Delta\phi^h(z))\right]\, dx\\
      &\quad + \frac{z}{2}\int_{\Omega} m_\varepsilon(\rho^h(z))m_\varepsilon^{\prime\prime}(\rho^h(z))\bigg[\nabla\varphi\cdot\nabla\rho^h(z)|\nabla\phi^h(z)|^2\\ 
       &\hspace{6cm} - 2\{\nabla\phi^h(z)\cdot\nabla\rho^h(z)\}\{\nabla\varphi\cdot\nabla\phi^h(z)\}\bigg]\, dx\\
      &\quad + z\int_{\Omega} m_\varepsilon(\rho^h(z))m_\varepsilon^\prime(\rho^h(z))\bigg[\nabla\varphi\cdot\nabla\left(\frac{1}{2}|\nabla\phi^h(z)|^2\right)
        - \nabla\phi^h(z)\cdot\nabla\{\nabla\varphi\cdot\nabla\phi^h(z)\}\bigg]\, dx\\
      &\quad +\frac{z\delta}{2}\int_{\partial\Omega} m_\varepsilon(\rho^h(z))\nabla(|\nabla\phi^h(z)|^2)\cdot\boldsymbol{n}\, dS\\
      &\eqcolon \mathrm{J}_1 + \mathrm{J}_2 + \mathrm{J}_3 + \mathrm{J}_4 + \mathrm{J}_5.
   \end{align*}
   First, since $m_\varepsilon^{\prime\prime} \leq 0$ ($m_\varepsilon$ is concave), it holds that
   \begin{align*}
      \mathrm{J}_1
      = -\frac{z\delta}{2}\int_{\Omega} |m_\varepsilon^{\prime\prime}(\rho^h(z))||\nabla\rho^h(z)|^2|\nabla\phi^h(z)|^2\, dx.
   \end{align*}
   Next, by the Bochner formula, it follows that
   \begin{equation*}
      -\Delta\left(\frac{1}{2}|\nabla\phi^h(z)|^2\right) + \nabla\phi^h(z)\cdot\nabla(\Delta\phi^h(z)) = -|\nabla^2\phi^h(z)|^2,
   \end{equation*}
   which implies
   \begin{align*}
      \mathrm{J}_2 
      = -z\delta\int_{\Omega} m_\varepsilon(\rho^h(z))|\nabla^2\phi^h(z)|^2\, dx,
   \end{align*}
   where $\nabla^2\phi^h(z)$ denotes the Hessian matrix of $\phi^h(z)$.
   Moreover, simple calculations yield 
   \begin{align*}
      &|\nabla\varphi\cdot\nabla\rho^h(z)|\nabla\phi^h(z)|^2 - 2\{\nabla\phi^h(z)\cdot\nabla\rho^h(z)\}\{\nabla\varphi\cdot\nabla\phi^h(z)\}|\\
      &\leq |\nabla\varphi||\nabla\phi^h(z)|^2|\nabla\rho^h(z)|\\
      &\leq 2|\nabla\varphi||\nabla\phi^h(z)|^2|\nabla\rho^h(z)|
   \end{align*}
   and
   \begin{align*}
      &\left|\nabla\varphi\cdot\nabla\left(\frac{1}{2}|\nabla\phi^h(z)|^2\right) - \nabla\phi^h(z)\cdot\nabla\{\nabla\varphi\cdot\nabla\phi^h(z)\}\right|\\
      &= |\nabla\varphi\cdot[\nabla^2\phi^h(z)\nabla\phi^h(z)] - \nabla\phi^h(z)\cdot[\nabla^2\varphi\nabla\phi^h(z) + \nabla^2\phi^h(z)\nabla\varphi]|\\ 
      &\leq \|\nabla^2\varphi\|_{L^\infty(\Omega)}|\nabla\phi^h(z)|^2.
   \end{align*}
   Hence $\mathrm{J}_3$ and $\mathrm{J}_4$ can be estimated as follows:
   \begin{align*}
      &\mathrm{J}_3
      \leq z\int_{\Omega} m_\varepsilon(\rho^h(z))|m_\varepsilon^{\prime\prime}(\rho^h(z))||\nabla\varphi||\nabla\phi^h(z)|^2|\nabla\rho^h(z)|\, dx,\\
      &\mathrm{J}_4
      \leq z\int_{\Omega} m_\varepsilon(\rho^h(z))|m_\varepsilon^\prime(\rho^h(z))|\|\nabla^2\varphi\|_{L^\infty(\Omega)}|\nabla\phi^h(z)|^2\, dx.
   \end{align*}
   Finally we control the boundary integral.
   It follows from Lemma \ref{lem2.1} that
   \begin{equation*}
      \nabla(|\nabla\phi^h(z)|^2)\cdot\boldsymbol{n}
      \leq C_{\Omega} |\nabla\phi^h(z)|^2\quad \mathrm{on}\ \partial\Omega.
   \end{equation*}
   Moreover using Lemma \ref{lem2.2}, we have
   \begin{align*}
      \|m_\varepsilon(\rho^h(z))^{\frac{1}{2}}|\nabla\phi^h(z)|\|_{L^2(\partial\Omega)}^2
      &\leq C_0^2\|m_\varepsilon(\rho^h(z))^{\frac{1}{2}}|\nabla\phi^h(z)|\|_{W^{s,2}(\Omega)}^2,
   \end{align*}
   where $\frac{1}{2} < s < 1$.
   It follows from the fractional Gagliardo--Nirenberg inequality (Lemma \ref{lem2.4})
   and Young's inequality that
   \begin{align*}
      &\|m_\varepsilon(\rho^h(z))^{\frac{1}{2}}|\nabla\phi^h(z)|\|_{W^{s,2}(\Omega)}^2\\
      &\leq 2C_1^2\|\nabla(m_\varepsilon(\rho^h(z))^{\frac{1}{2}}|\nabla\phi^h(z)|)\|_{L^2(\Omega)}^{2s}\|m_\varepsilon(\rho^h(z))^{\frac{1}{2}}|\nabla\phi^h(z)|\|_{L^2(\Omega)}^{2(1-s)}\\
      &\hspace{7cm} + 2C_2^2\|m_\varepsilon(\rho^h(z))^{\frac{1}{2}}|\nabla\phi^h(z)|\|_{L^2(\Omega)}^2\\
      &\leq \sigma\|\nabla(m_\varepsilon(\rho^h(z))^{\frac{1}{2}}|\nabla\phi^h(z)|)\|_{L^2(\Omega)}^2 + \sigma^{-1}\bar{C}\|m_\varepsilon(\rho^h(z))^{\frac{1}{2}}|\nabla\phi^h(z)|\|_{L^2(\Omega)}^2
   \end{align*}
   for $\sigma > 0$,
   where $\bar{C}>0$ is a constant.
   Here since $\rho^h(z)$ and $\phi^h(z)$ are smooth, we can calculate 
   \begin{align*}
      \nabla(m_\varepsilon(\rho^h(z))^{\frac{1}{2}}|\nabla\phi^h(z)|)
      = \frac{m_\varepsilon^\prime(\rho^h(z))\nabla\rho^h(z)}{2m_\varepsilon(\rho^h(z))^{\frac{1}{2}}}|\nabla\phi^h(z)|
       + m_\varepsilon(\rho^h(z))^{\frac{1}{2}}\nabla(|\nabla\phi^h(z)|)\quad \mathrm{a.e.\ in}\ \Omega.
   \end{align*}
   Hence we infer from Lemma \ref{lem2.5} that
   \begin{align*}
      \|\nabla(m_\varepsilon(\rho^h(z))^{\frac{1}{2}}|\nabla\phi^h(z)|)\|_{L^2(\Omega)}^2
      &\leq 2\int_{\Omega} \frac{(m_\varepsilon^\prime(\rho^h(z)))^2}{4m_\varepsilon(\rho^h(z))}|\nabla\rho^h(z)|^2|\nabla\phi^h(z)|^2\, dx\\
      &\quad + 2\int_{\Omega} m_\varepsilon(\rho^h(z))|\nabla^2\phi^h(z)|^2\, dx.
   \end{align*}
   In the end, since $m_\varepsilon$ satisfies the condition \eqref{M-A}, 
   the boundary integral $\mathrm{J}_5$ is dominated by
   \begin{align*}
      \mathrm{J}_5
      &\leq \frac{z\delta\tilde{C}_\Omega\sigma}{4}\sup_{0<r<M}\frac{(m_\varepsilon^\prime(r))^2}{m_\varepsilon(r)|m_\varepsilon^{\prime\prime}(r)|}
       \int_{\Omega} |m_\varepsilon^{\prime\prime}(\rho^h(z))||\nabla\rho^h(z)|^2|\nabla\phi^h(z)|^2\, dx\\
      &\quad + z\delta\tilde{C}_\Omega\sigma\int_{\Omega} m_\varepsilon(\rho^h(z))|\nabla^2\phi^h(z)|^2\, dx\\
      &\quad + \frac{z\delta\tilde{C}_\Omega\bar{C}}{\sigma}\int_{\Omega} m_\varepsilon(\rho^h(z))|\nabla\phi^h(z)|^2\, dx,
   \end{align*}
   where $\tilde{C}_\Omega = C_\Omega C_0^2>0$.
   Using the above estimates and Young's inequality, we obtain
   \begin{align*}
      &\frac{1}{2}\dhh\boldsymbol{A}^h(z) + \partial_zF(\rho^h(z))\\
      &\leq \mathrm{J}_1 + \mathrm{J}_2 + \mathrm{J}_3 + \mathrm{J}_4 + \mathrm{J}_5\\
      &\leq -\frac{z\delta}{2}\int_{\Omega} |m_\varepsilon^{\prime\prime}(\rho^h(z))||\nabla\rho^h(z)|^2|\nabla\phi^h(z)|^2\, dx\\
      &\quad -z\delta\int_{\Omega} m_\varepsilon(\rho^h(z))|\nabla^2\phi^h(z)|^2\, dx\\
      &\quad + z\int_{\Omega} \bigg[\left(\frac{1}{2}\delta\right)^{\frac{1}{2}}|m_\varepsilon^{\prime\prime}(\rho^h(z))|^{\frac{1}{2}}|\nabla\phi^h(z)||\nabla\rho^h(z)|\\
       &\hspace{3cm} \times \left(\frac{1}{2}\delta\right)^{-\frac{1}{2}}m_\varepsilon(\rho^h(z))|m_\varepsilon^{\prime\prime}(\rho^h(z))|^{\frac{1}{2}}|\nabla\varphi||\nabla\phi^h(z)|\bigg]\, dx\\
      &\quad + z\int_{\Omega} m_\varepsilon(\rho^h(z))|m_\varepsilon^\prime(\rho^h(z))|\|\nabla^2\varphi\|_{L^\infty(\Omega)}|\nabla\phi^h(z)|^2\, dx\\
      &\quad + \frac{z\delta \tilde{C}_\Omega\sigma}{4}\sup_{0<r<M}\frac{(m_\varepsilon^\prime(r))^2}{m_\varepsilon(r)|m_\varepsilon^{\prime\prime}(r)|}
       \int_{\Omega} |m_\varepsilon^{\prime\prime}(\rho^h(z))||\nabla\rho^h(z)|^2|\nabla\phi^h(z)|^2\, dx\\
      &\quad + z\delta\tilde{C}_\Omega\sigma\int_{\Omega} m_\varepsilon(\rho^h(z))|\nabla^2\phi^h(z)|^2\, dx\\
      &\quad + \frac{z\delta\tilde{C}_\Omega\bar{C}}{\sigma}\int_{\Omega} m_\varepsilon(\rho^h(z))|\nabla\phi^h(z)|^2\, dx\\
      &\leq -\frac{z\delta}{4}\left[2 - 1 - \tilde{C}_\Omega\sigma\sup_{0<r<M}\frac{(m_\varepsilon^\prime(r))^2}{m_\varepsilon(r)|m_\varepsilon^{\prime\prime}(r)|}\right]\int_{\Omega} |m_\varepsilon^{\prime\prime}(\rho^h(z))||\nabla\rho^h(z)|^2|\nabla\phi^h(z)|^2\, dx\\
      &\quad - z\delta\left[1 - \tilde{C}_\Omega\sigma\right]\int_{\Omega} m_\varepsilon(\rho^h(z))|\nabla^2\phi^h(z)|^2\, dx\\
      &\quad + \frac{z}{\delta}\int_{\Omega} |m_\varepsilon(\rho^h(z))|^2|m_\varepsilon^{\prime\prime}(\rho^h(z))||\nabla\varphi|^2|\nabla\phi^h(z)|^2\, dx\\
      &\quad + z\int_{\Omega} m_\varepsilon(\rho^h(z))|m_\varepsilon^\prime(\rho^h(z))|\|\nabla^2\varphi\|_{L^\infty(\Omega)}|\nabla\phi^h(z)|^2\, dx\\
      &\quad + \frac{z\delta\tilde{C}_\Omega\bar{C}}{\sigma}\int_{\Omega} m_\varepsilon(\rho^h(z))|\nabla\phi^h(z)|^2\, dx.
   \end{align*}
   If $\sigma > 0$ satisfies 
   \begin{align*}
      \sigma < \frac{1}{\tilde{C}_\Omega}\min\left\{1,\left(\sup_{0<r<M}\frac{(m_\varepsilon^\prime(r))^2}{m_\varepsilon(r)|m_\varepsilon^{\prime\prime}(r)|}\right)^{-1}\right\},
   \end{align*} 
   then by the assumption \eqref{M-LSC}, it holds that
   \begin{align*}
      &\frac{1}{2}\dhh\boldsymbol{A}^h(z) + \partial_zF(\rho^h(z))\\
      &\leq \frac{z}{\delta}\int_{\Omega} |m_\varepsilon(\rho^h(z))|^2|m_\varepsilon^{\prime\prime}(\rho^h(z))||\nabla\varphi|^2|\nabla\phi^h(z)|^2\, dx\\
      &\quad + z\int_{\Omega} m_\varepsilon(\rho^h(z))|m_\varepsilon^\prime(\rho^h(z))|\|\nabla^2\varphi\|_{L^\infty(\Omega)}|\nabla\phi^h(z)|^2\, dx\\
      &\quad + \frac{z\delta\tilde{C}_\Omega\bar{C}}{\sigma}\int_{\Omega} m_\varepsilon(\rho^h(z))|\nabla\phi^h(z)|^2\, dx\\
      &\leq - z\left(-\frac{\|\nabla\varphi\|_{L^\infty(\Omega)}^2}{\delta}\sup_{0<r<M}(m_\varepsilon(r)|m_\varepsilon^{\prime\prime}(r)|) - \|\nabla^2\varphi\|_{L^\infty(\Omega)}\sup_{0<r<M}|m_\varepsilon^\prime(r)| - \frac{\delta\tilde{C}_\Omega\bar{C}}{\sigma}\right)\boldsymbol{A}^h(z).
   \end{align*}
   Therefore, putting
   \begin{align*}
      \lambda_{\delta} 
      &\coloneq -\frac{\|\nabla\varphi\|_{L^\infty(\Omega)}^2}{\delta}\sup_{0<r<M}(m_\varepsilon(r)|m_\varepsilon^{\prime\prime}(r)|) 
       - \|\nabla^2\varphi\|_{L^\infty(\Omega)}\sup_{0<r<M}|m_\varepsilon^\prime(r)|
       - \frac{\delta C}{\sigma} \leq 0
   \end{align*}
   where $C = \tilde{C}_\Omega\bar{C}>0$, we conclude
   \begin{align*}
      \frac{1}{2}\dhh\boldsymbol{A}^h(z) + \partial_zF(\rho^h(z)) 
      \leq -\lambda_{\delta} z \boldsymbol{A}^h(z).
   \end{align*}
   The proof is completed.
\end{proof}

\begin{cor}[The EVI for the weighted Wasserstein distance]\label{cor3.4}
   Under the same assumptions of Theorem \ref{thm3.1},
   the evolution variational inequality for the weighted Wasserstein distance holds:
   \begin{align*}
      \frac{1}{2}\limsup_{h\downarrow0}\frac{W_{m_\varepsilon}(S_h(u),v)^2 - W_{m_\varepsilon}(u,v)^2}{h}
      + \frac{\lambda_\delta}{2}W_{m_\varepsilon}(u,v)^2
      \leq F(v) - F(u),
   \end{align*}
   for $u,v \in L^1\cap\mathcal{P}(\Omega)$ with $W_{m_\varepsilon}(u,v), F(u), F(v) < \infty$,
   where $\lambda_\delta$ is defined in Theorem \ref{thm3.1}.
\end{cor}

\begin{proof}
   Combining the arguments in \cite[Lemma 3.3]{LMS} or \cite[Lemma 4.8]{M} with Theorem \ref{thm3.1},
   we complete the proof.
\end{proof}

\section{Application}

\quad In this section,
we apply Theorem \ref{thm3.1} (Corollary \ref{cor3.4}) to the minimizing movement in weighted Wasserstein metrics 
for specific equations.
The applications are the global existence of weak solutions
to the equations, which have the gradient flow structures in weighted Wasserstein metrics,
such as Keller--Segel systems and Cahn--Hilliard type equations 
in bounded non-convex domains.

\subsection{Quasilinear Keller--Segel systems}

\quad We consider the following quasilinear Keller--Segel system with nonlinear mobility:
\begin{align}\label{kspp}
    \begin{cases}
      \partial_t u = \Delta u^p - \nabla\cdot(\chi u^\alpha \nabla v)
       &\mathrm{in}\ \Omega\times(0,\infty),\\
      \dt v = \Delta v - v + u 
       &\mathrm{in}\ \Omega\times(0,\infty),\\
      \nabla u\cdot \boldsymbol{n} = \nabla v\cdot \boldsymbol{n} = 0
       &\mathrm{on}\ \partial\Omega\times(0,\infty),\\
      u(\cdot,0) = u_0(\cdot),\ v(\cdot,0) = v_0(\cdot)
       &\mathrm{in}\ \Omega,
    \end{cases}
\end{align}
where 
$\Omega$ is a bounded domain in $\Rd$ with the smooth boundary,
$d\geq2$, $p\geq1$, $0<\alpha<1$, $\chi>0$ and $\boldsymbol{n}$ is the outer unit normal vector to $\partial\Omega$.
In addition, 
$u_0 \in L^{p+1-\alpha}\cap\mathcal{P}(\Omega)$ is a nonnegative function
and $v_0 \in H^1(\Omega)$ is also a nonnegative function.

\quad The Keller--Segel system is a model to describe an aggregation phenomenon of cellular slime molds with chemotaxis.
In particular, the system \eqref{kspp} takes into account a volume-filling effect
by assuming $0<\alpha<1$ (if $\alpha=1$ then \eqref{kspp} is the well-known Keller--Segel system).
Here, the volume-filling effect is a phenomenon that movements of cells are restricted by presence of other cells
(weaker aggregation).

\quad There are some results about the existence of solutions to \eqref{kspp}.
For instance, in \cite{ISY}, the authors proved the existence of global weak solutions to \eqref{kspp}
in the sub-critical case $p>1+\alpha-2/d$ when $\Omega$ is a bounded non-convex domain.
Note that they did not use the minimizing movement in weighted Wasserstein metrics.
On the other hand, in \cite{M}, it is shown that there exist global weak solutions to \eqref{kspp}
in the case $p\geq1+\alpha-2/d$ (assuming $\chi>0$ is small enough if $p=1+\alpha-2/d$) 
by the minimizing movement in weighted Wasserstein metrics when $\Omega$ is a bounded convex domain.
Here the assumption of convexity of the domain comes from the evolution variational inequality.
Thus, combining Theorem \ref{thm3.1} with the argument of the minimizing movement 
in the weighted Wasserstein space in \cite{M},
we can prove the existence of weak solutions to the degenerate Keller--Segel system \eqref{kspp} 
in the case $p\geq1+\alpha-2/d$ (and $\chi$ is small enough if $p=1+\alpha-2/d$) 
without assuming the convexity of $\Omega$.
Moreover, we emphasize that this existence result (degenerate, the critical case and the non-convexity) is new.

\begin{thm}\label{thm4.1}
   Let $1+\alpha -2/d < p \leq 1+\alpha$ or
   assume $p=1+\alpha-2/d$ and $\chi>0$ is small enough.
   Then for all $T>0$, there exists a nonnegative weak solution $(u,v)$ to \eqref{kspp}
   on the time interval $[0,T]$ satisfying
   \begin{align*}
      &\bullet u \in L^\infty((0,T);L^{p+1-\alpha}(\Omega)),\ u^{\frac{p+1-\alpha}{2}} \in L^2((0,T);H^1(\Omega)),\\
      &\bullet \|u(t)\|_{L^1(\Omega)} = 1\quad \mathrm{for}\ t \in [0,T],\\
      &\bullet v \in L^\infty((0,T);H^1(\Omega))\cap L^2((0,T);H^2(\Omega))\cap C^{\frac{1}{2}}([0,T];L^2(\Omega)),\\
      &\bullet \lim_{t\to0}W_m(u(t),u_0) = 0\ \mathrm{and}\ \lim_{t\to0}\|v(t) - v_0\|_{L^2(\Omega)} = 0,
   \end{align*}
   where $W_m$ is the weighted Wasserstein distance 
   and $m(u) = u^\alpha$, moreover
   \begin{align*}
      &\int_0^T\int_{\Omega} (\nabla u^p - \chi u^\alpha\nabla v)\cdot\nabla\varphi\, dx\,dt 
       = \int_{\Omega} (u_0 - u(\cdot,T))\varphi\, dx,\\
      &\int_0^T\int_{\Omega} \left[\nabla v\cdot\nabla\zeta + v\zeta - u\zeta\right]\, dx\, dt
       = \int_{\Omega} (v_0 - v(\cdot,T))\zeta\, dx,
   \end{align*}
   for all $\varphi \in C^\infty(\overline{\Omega})$ 
   with $\nabla \varphi\cdot\boldsymbol{n} = 0$ on $\partial\Omega$
   and $\zeta \in H^1(\Omega)$.
\end{thm}

\begin{proof}
   The proof is almost same as the one in \cite{M} except for \cite[Lemma 4.4]{M}.
   Substituting Theorem \ref{thm3.1} for it,
   we can complete the proof.
   Thus, to apply Theorem \ref{thm3.1}, 
   let us check that the approximated mobility $m_\varepsilon(u)=(u+\varepsilon)^\alpha$
   satisfies the conditions \eqref{M-LSC} and \eqref{M-A}.
   Indeed,
   \begin{align*}
      &\sup_{r\geq0}|m_\varepsilon^\prime(r)| 
      = \sup_{r\geq0}\{\alpha(r+\varepsilon)^{\alpha-1}\} 
      = \frac{\alpha}{\varepsilon^{1-\alpha}},\\
      &\sup_{r\geq0}|m_\varepsilon^{\prime\prime}(r)m_\varepsilon(r)| 
      = \sup_{r\geq0}\{\alpha(1-\alpha)(r+\varepsilon)^{2(\alpha-1)}\} 
      = \frac{\alpha(1-\alpha)}{\varepsilon^{2(1-\alpha)}},\\
      &\sup_{r\geq0}\frac{(m_\varepsilon^\prime(r))^2}{|m_\varepsilon^{\prime\prime}(r)m_\varepsilon(r)|} 
      = \frac{\alpha}{1-\alpha}.
   \end{align*}
   The proof is completed.
\end{proof}

\subsection{Cahn--Hilliard type equations}

\quad We consider the following Cahn--Hilliard type equation with nonlinear mobility, 
which is considered in \cite{LMS}:
\begin{align}\label{ch}
   \begin{cases}
      \partial_t u = - \nabla\cdot(m(u)\nabla(\Delta u - G^\prime(u)))\quad &\mathrm{in}\ \Omega\times(0,\infty),\\
      \nabla u\cdot\boldsymbol{n} = 0, \quad (m(u)\nabla(\Delta u - G^\prime(u)))\cdot\boldsymbol{n} = 0 &\mathrm{on}\ \partial\Omega\times(0,\infty),\\
      u(\cdot,0) = u_0(\cdot) &\mathrm{in}\ \Omega,
   \end{cases}
\end{align}
where $\Omega$ is a bounded domain in $\mathbb{R}^d$ with the smooth boundary,
$d\geq1$, $\boldsymbol{n}$ is the outer unit normal vector to $\partial\Omega$
and the mobility $m$ satisfies the conditions in Section 3.1.
In addition, 
concering the free energy $G \in C^2(0,M)$ and the associated pressure $P$ satisfying
\begin{align*}
   P^\prime(s) = m(s)G^{\prime\prime}(s),
\end{align*}
we assume that
there exist a constant $C\geq0$ and an exponent $q > 2$ with $q < 2d/(d-4)$ if $d>4$ such that 
\begin{align}\label{G}
   \begin{cases}
      mG^{\prime\prime} \geq -C,\quad P \in C([0,M])\quad &\mathrm{if}\ M < \infty,\\
      \displaystyle
      mG^{\prime\prime} \geq -C(1+m)\quad \mathrm{in}\ (0,\infty), 
      \quad P \in C([0,\infty)),\quad \lim_{s\to\infty}\frac{P(s)}{s^q + |G(s)|} = 0
      &\mathrm{if}\ M = \infty.
   \end{cases}\tag{G}
\end{align}

\quad Equations of the form \eqref{ch} arise as hydrodynamic approximation to models 
for many-particle systems in gas dynamics and also in lubrication theory.
For instance (see also \cite{CH, TC}), by choosing $m(r) = r(1-r)$ (the case $M = 1 <\infty$) and $G(r) = \theta r^2(1-r)^2$ for $\theta \in \mathbb{R}$,
the equation \eqref{ch} is a model for the phase separation for a binary alloy:
\begin{align*}
   \partial_t u = -\nabla\cdot(u(1-u)\nabla\Delta u) + \theta\Delta u^2(1-u)^2.
\end{align*}
In addition, if $G(r) = \theta(r\log r + (1-r)\log (1-r))$ 
then we have an equation for the volume fraction of one component in binary gas mixture:
\begin{align*}
   \partial_t u = - \nabla\cdot(u(1-u)\nabla\Delta u) + \theta\Delta u.
\end{align*}
On the other hand, by choosing $m(r) = r^\alpha$ with $1/2 < \alpha \leq 1$ (the case $M=\infty$) and
\begin{align*}
   G(r) = \kappa\frac{\beta}{(\beta-\alpha)(\beta - \alpha + 1)}r^{\beta-\alpha+1}
\end{align*}
for $\kappa \in \mathbb{R}$ and $1\leq \beta \leq \alpha +1$,
the equation \eqref{ch} is the thin film equation:
\begin{align*}
   \partial_t u = - \nabla\cdot(u^\alpha\nabla\Delta u) + \kappa\Delta(u^\beta).
\end{align*}

\quad In \cite{LMS}, the authors proved the existence of weak solutions to \eqref{ch} in a bounded convex domain
under the above assumptions by the minimizing movement in weighted Wasserstein metrics.
On the other hand, in \cite{EG}, the related existence result is proved in a bounded domain (not convex)
under the similar assumptions.
Note that the minimizing movement is not used in \cite{EG}.
Then combining Theorem \ref{thm3.1} with the argument of the minimizing movement in the weighted Wassrstein space in \cite{LMS},
we can show the existence of weak solutions to \eqref{ch} in a bounded non-convex domain by the gradient flow approach.
This also means that Theorem \ref{thm3.1} can improve the result in \cite{LMS}.

\begin{thm}\label{thm4.2}
   For the mobility $m$ in Section 3.1, assume that $m$ satisfies
   \begin{align}
      \lim_{r\downarrow0}r^{\frac{1}{2}}m^\prime(r) = 0\quad \mathrm{and}\ 
      \lim_{r\uparrow M}(M-r)^{\frac{1}{2}}m^\prime(r) = 0\quad \mathrm{if}\ M<\infty
   \end{align}
   and $G$ satisfies \eqref{G}.
   Let $u_0 \in H^1\cap\mathcal{P}(\Omega)$ such that
   \begin{align*}
      0 \leq u_0 \leq M\ \mathrm{a.e.\ in}\ \Omega,\quad E(u_0) < \infty,
   \end{align*}
   where $E$ is the energy functional defined by
   \begin{align*}
      E(v) = \frac{1}{2}\int_{\Omega} |\nabla v|^2\, dx + \int_{\Omega} G(v)\, dx.
   \end{align*}
   Then there exists a weak solution $u$ to \eqref{ch} satisfying the following:
   \begin{align*}
      &u \in L_{loc}^2([0,\infty);H^2(\Omega))\cap C_w([0,\infty);H^1(\Omega)),\\
      &\|u(t)\|_{L^1(\Omega)} = \|u_0\|_{L^1(\Omega)},\ 0\leq u(t) \leq M\ \mathrm{a.e.\ in}\ \Omega\quad \mathrm{for}\ t\geq0,\\
      &E(u(t)) \leq E(u_0)\quad \mathrm{for\ all}\ t\geq0,
   \end{align*}
   and
   \begin{align*}
      \int_0^\infty\int_{\Omega} u\partial_t \varphi\, dx\, dt 
      = \int_0^\infty\int_{\Omega} \Delta u \nabla\cdot(m(u)\nabla\varphi)\, dx\, dt
      - \int_0^\infty\int_{\Omega} P(u)\Delta\varphi\, dx\, dt
   \end{align*}
   for every test function $\varphi \in C_c^\infty(\bar{\Omega}\times(0,\infty))$ 
   such that $\nabla\varphi\cdot\boldsymbol{n} = 0$ on $\partial\Omega$.
\end{thm}

\begin{rem}
   $C_w(([0,\infty);H^1(\Omega)))$ denotes the space of weakly continuous curves 
   $u : [0,\infty) \to H^1(\Omega)$, that is, if $u \in C_w([0,\infty);H^1(\Omega))$ then
   for $t \in [0,\infty)$, it holds
   \begin{align*}
      \langle\zeta, u(s)\rangle \to \langle\zeta, u(t)\rangle
      \quad \mathrm{as}\ s\to t\quad \mathrm{for\ all}\ \zeta \in (H^1(\Omega))^*.
   \end{align*}
\end{rem}

In the proof of \cite{LMS}, the convexity of the domain is also used 
when a Sobolev like inequality is established (see \cite[Lemma A.1]{LMS}).
However, employing the same method of the proof of Theorem \ref{thm3.1},
we can show the alternative inequality, which is enough to prove Theorem \ref{thm4.2}.

\begin{lem}\label{lem4.7}
   Assume that $\Omega$ is a bounded domain and 
   that $f\in H^2(\Omega)$ satisfies $\nabla f \cdot\boldsymbol{n} = 0$ on $\partial\Omega$.
   Then there exists a positive constant $C>0$ depending on $\Omega$ such that
   \begin{align}\label{eqA}
      \int_{\Omega} |\nabla^2 f|^2\, dx \leq 2\int_{\Omega} (\Delta f)^2\, dx + C\|f\|_{H^1(\Omega)}^2.
   \end{align}
\end{lem}

\begin{proof}
   By density, it suffies to prove \eqref{eqA} for $f \in C^\infty(\bar{\Omega})$ 
   satisfying $\nabla f\cdot\boldsymbol{n} = 0$ on $\partial\Omega$.
   It follows from integration by parts that
   \begin{align*}
      \int_{\Omega} |\nabla^2 f|^2\, dx
      &= \int_{\partial\Omega} \nabla^2 f\nabla f\cdot\boldsymbol{n} \, dS 
       - \int_\Omega \nabla\Delta f\cdot \nabla f\, dx\\
      &= \int_{\partial\Omega} \nabla^2 f\nabla f\cdot\boldsymbol{n}\, dS 
       - \int_{\partial\Omega} \Delta f\nabla f\cdot\boldsymbol{n}\, dS 
       + \int_{\Omega} (\Delta f)^2\, dx.
   \end{align*}
   Since $\nabla f\cdot\boldsymbol{n} = 0$ on $\partial\Omega$, 
   the second term vanishes.
   In addition, by Lemmas \ref{lem2.1}, \ref{lem2.2} and \ref{lem2.4},
   the first term can be estimated as follows:
   \begin{align*}
      \int_{\partial\Omega} \nabla^2 f\nabla f\cdot\boldsymbol{n} \, dS 
      &= \int_{\partial\Omega} \frac{1}{2}\nabla(|\nabla f|^2)\cdot\boldsymbol{n} \, dS\\
      &\leq \frac{C_\Omega}{2} \int_{\partial\Omega} |\nabla f|^2\, dS = \frac{C_\Omega}{2}\|\nabla f\|_{L^2(\partial\Omega)}^2\\
      &\leq \frac{C_\Omega}{2}C_0^2 \|\nabla f\|_{W^{s,2}(\Omega)}^2\quad \left(\frac{1}{2} < s < 1\right)\\
      &\leq C_\Omega C_0^2 \left(C_1^2\|\nabla|\nabla f|\|_{L^2(\Omega)}^{2s}\|\nabla f\|_{L^2(\Omega)}^{2(1-s)} + C_2^2\|\nabla f\|_{L^2(\Omega)}^2\right),
   \end{align*}
   Using Young's inequality, we obtain
   \begin{align*}
      \int_{\partial\Omega} \nabla^2 f\nabla f\cdot\boldsymbol{n} \, dS 
      \leq \frac{1}{2}\|\nabla|\nabla f|\|_{L^2(\Omega)}^2 + C\|\nabla f\|_{L^2(\Omega)}^2,
   \end{align*}
   where $C>0$ is a constant depending on $\Omega$.
   Finally, it follows from Lemma \ref{lem2.5} that
   \begin{align*}
      \int_{\partial\Omega} \nabla^2 f\nabla f\cdot\boldsymbol{n} \, dS 
      \leq \frac{1}{2}\int_{\Omega} |\nabla^2 f|^2\, dx + C\|f\|_{H^1(\Omega)}^2.
   \end{align*}
   The proof is completed.
\end{proof}

\begin{proof}[Proof of Theorem \ref{thm4.2}]
   The proof is almost same as the one in \cite{LMS} except for some arguments.
   First, we use Theorem \ref{thm3.1} (Corollary \ref{cor3.4}) in the proof of \cite[Lemma 4.2 and Proposition 4.6]{LMS}.
   Moreover, the estimate (61) in \cite[Proposition 3.1]{LMS} follows from \cite[Lemma 6.3]{M}, 
   which is the refined Ascoli--Arzel\`a theorem for the approximation from above.
   Finally, substituting Lemma \ref{lem4.6} and Lemma \ref{lem4.7} 
   for the lower bound in \cite[Lemma A.1]{LMS} and \cite[Lemma 5.1]{LMS} respectively,
   we complete the proof.
\end{proof}


\end{document}